\documentclass[12pt, a4paper]{article}
\usepackage{amsfonts}
\usepackage{mathrsfs}
\usepackage{latexsym}
\usepackage{xy}
\usepackage{amsfonts,amsmath,amssymb,amsthm}
\usepackage{color}

\xyoption{all}

\newcommand{\bcen}{\begin{center}}     \newcommand{\ecen}{\end{center}}
\newcommand{\bay}{\begin{array}}      \newcommand{\eay}{\end{array}}
\newcommand{\beq}{\begin{eqnarray*}}      \newcommand{\eeq}{\end{eqnarray*}}
\def\add{\mathrm{add}}

\def\End{\mathrm{End}}

\def\Ext{\mathrm{Ext}}

\def\gl{\mathrm{gl.dim}}

\def\Hom{\mathrm{Hom}}

\def\Ker{\mathrm{Ker}}
\def\mod{\mathrm{mod}}
\def\Mod{\mathrm{Mod}}

\def\pd{\mathrm{pd}}

\def\proj{\mathrm{proj}}
\def\Proj{\mathrm{Proj}}
\def\Gproj{\mathrm{Gproj}}

\begin{document}

\newtheorem{theorem}{Theorem}
\newtheorem{proposition}{Proposition}
\newtheorem{lemma}{Lemma}
\newtheorem{corollary}{Corollary}
\newtheorem{remark}{Remark}
\newtheorem{example}{Example}
\newtheorem{definition}{Definition}
\newtheorem*{conjecture}{Conjecture}
\newtheorem{question}{Question}

\title{\large\bf Recollements, Cohen-Macaulay Auslander algebras and Gorenstein projective conjecture}

\author{\large Yongyun Qin}

\date{\footnotesize College of Mathematics and Statistics,
Qujing Normal University, \\ Qujing, Yunnan 655011, China. E-mail:
qinyongyun2006@126.com
}

\maketitle

\begin{abstract}
It is shown that a $4$-recollement of derived categories
of CM-finite algebras induces a $2$-recollement of the corresponding
Cohen-Macaulay Auslander algebras, which
generalises the main theorem of Pan [S. Y. Pan, Derived equivalences for Cohen-Macaulay Auslander algebras,
J. Pure Appl. Algebra 216 (2012), 355--363].
Moreover, both Auslander-Reiten conjecture and Gorenstein projective conjecture are
shown invariant under
$3$ (or $4$)-recollement of unbounded derived categories of algebras.
\end{abstract}

\medskip


\medskip

{\footnotesize {\bf Keywords} : derived category; recollement;
Cohen-Macaulay Auslander algebra,
\linebreak
Auslander-Reiten conjecture, Gorenstein projective conjecture. }

\section{\large Introduction}
\indent\indent Recollements of triangulated categories, introduced by Beilinson et al.
\cite{BBD82}, are widely used in algebraic geometry and
representation theory. A recollement of triangulated categories
describes one category
as being ``glued together'' from two others, where these three categories
were related by sixes adjoint functors. In particular, a recollement
is called an $n$-recollement
if there are more than sixes functors \cite{QH16}, and the collections of adjacent recollements
are called ladders \cite{BGS88,AKLY17}. $n$-recollements of unbounded
derived categories of algebras
unify the recollements of different kinds of bounded derived categories:
a recollement of $\mathcal{D}^-(\mod)$ is a $2$-recollement
of $\mathcal{D}(\Mod )$, a recollement of $K^b(\proj)$ is a $3$-recollement
of $\mathcal{D}(\Mod )$, and a recollement of $\mathcal{D}^b(\mod)$ is a $3$-recollement
of $\mathcal{D}(\Mod )$ if the algebras are
finite-dimensional over a field $k$.
On the other hand, the language of n-recollements has proved
powerful in clarifying the relationships between
recollements of derived categories and certain homological properties of algebras
\cite{QH16}. We refer to \cite{GP16,ZZZZ16} for more
related topics on ladders and $n$-recollements.

In \cite{Pan12}, Pan proved that if $A$ and $B$ are CM-finite Gorenstein algebras
that are derived equivalent, then their Cohen-Macaulay
Auslander algebras are also derived equivalent.
Later, Pan and Zhang obtained a similar result for two CM-finite algebras that are
standard derived equivalent \cite{PZ15}.
In this paper, we generalise their results in terms of recollements.

{\bf Theorem A } {\it
Let $A$, $B$ and $C$ be Artin $R$-algebras such that
$B$ and $C$ are CM-finite. If
$\mathcal{D}(\Mod A)$ admits a $4$-recollement relative to
$\mathcal{D}(\Mod B)$ and $\mathcal{D}(\Mod C)$,
then $\mathcal{D}(\Mod H)$
admits a $2$-recollement relative to
$\mathcal{D}(\Mod \Gamma _C)$ and $\mathcal{D}(\Mod \Gamma _B)$,
where $\Gamma _B$ and $\Gamma _C$ are the corresponding Cohen-Macaulay Auslander algebras,
and $H$ is an algebra constructed from the given recollement
(see Theorem~\ref{theorem-1} for detail).}

The second part of this paper is on recollements
and homological conjectures. Here, we focus on
Auslander-Reiten conjecture (ARC) and Gorenstein projective conjecture (GPC).
ARC is invariant under tilting equivalences \cite{Wei10} and
general derived equivalences \cite{Pan14}, and GPC is
invariant under standard derived equivalences \cite{PZ15}.
So, it is natural to investigate the invariance of these conjectures along recollements
of derived categories. In this paper, we get the following theorem
which is a combination of Theorem~\ref{theorem-3-GPC-ARC}, Theorem~\ref{theorem-3},
Corollary~\ref{corollary-4-ARC-GPC} and Theorem~\ref{theorem-4}.

{\bf Theorem B} {\it
Let $A$, $B$ and $C$ be Artin $R$-algebras, and $\mathcal{D}A$ admit an $n$-recollement relative
to $\mathcal{D}B$ and $\mathcal{D}C$.

{\rm (1)} $n=3:$ If $A$ satisfies ARC (resp. GPC), then so does $B$; If $\gl C<\infty$, then
$A$ satisfies ARC (resp. GPC)
if and only if so does $B$.

{\rm (2)} $n=4:$ If $A$ satisfies ARC (resp. GPC), then so do $B$ and $C$;
If $C$
is CM-free (resp. $\underline{^\bot C}=0$) and $B$ satisfies GPC (resp. ARC),
then $A$ satisfies GPC (resp. ARC), and the statement is also
true if we exchange the roles of $B$ and $C$. }

In \cite{PZ15}, the author asked whether GPC is
invariant under any derived equivalence. Due to Theorem B,
we give a positive answer. Moreover, we will use an example to
show how Theorem B can be used.

In \cite{AKLY17}, the authors showed that there exists
a $1$-recollement which can not be extended to a $2$-recollement,
a $2$-recollement but not a $3$-recollement,
and a $3$-recollement but not a $4$-recollement.
So, some experts ask whether there exists a $4$-recollement
but not a $5$-recollement, or,
any $4$-recollement can always be extended to a $\infty$-recollement.
At the end of this paper, we obtain a
$4$-recollement but not
a $5$-recollement, which may be useful in the further
development of ladders and $n$-recollements.

The paper is organized as follows: Section 2 recalls some
relevant definitions and conventions.
Section 3 and section 4 are the proof of Theorem A
and Theorem B respectively. The last section are
two examples, one is a
$4$-recollement but not
a $5$-recollement, and the other is reducing Auslander-Reiten conjecture by
recollement.

\section{\large Notations and definitions}
\indent\indent Let $A$ be an Artin algebra. Denote by $\Mod A$ the
category of right $A$-modules, and by $\Proj A$ (resp. $\mod A$ and
$\proj A$) its full subcategory consisting of all projective modules
(resp. finitely
generated modules and finitely
generated projective modules).
For $* \in \{{\rm nothing}, -, +, b \}$, denote by $\mathcal{D}^*(\Mod
A)$ (resp. $\mathcal{D}^*(\mod A)$) the derived category of cochain complexes of objects in $\Mod
A$ (resp. $\mod A$) satisfying the corresponding boundedness condition. Let
$K^b(\proj A)$ be the homotopy category of
bounded complexes of objects in $\proj A$.
We often view it as a full subcategory of $\mathcal{D}(\Mod A)$ and identify it with its essential image.
Up to isomorphism, the objects in $K^{b}(\proj A)$ are
precisely all the compact objects in $\mathcal{D}(\Mod A)$.
Usually, we
just write $\mathcal{D} A$ instead of $ \mathcal{D}(\Mod A)$.

From \cite{Buc87}, the singularity category $D_{sg}(A)$ of $A$ is defined to be
the Verdier quotient $D^b(\mod A)/K^b(\proj A)$.

A complete $A$-projective resolution is an exact sequence of projective $A$-modules
$$(P^\bullet,d) : \cdots \longrightarrow P^{-1} \longrightarrow
P^{0} \longrightarrow P^{1} \longrightarrow
\cdots $$
such that $\Hom _A(P^\bullet,A)$ is also exact. An $A$-module $M$
is {\it Gorenstein-projective}, if there exists a complete $A$-projective resolution $P^\bullet$
such that $M =\Ker d_0$.
Denote by $\Gproj A$ the full subcategory of $\mod A$ consisting
of all Gorenstein-projective modules, and $K^b(\Gproj A)$ the homotopy category of
bounded complexes of objects in $\Gproj A$.
By Buchweitz's Theorem,
$\Gproj A$ is a Frobenius category with the projective modules as the projective-injective
objects, so the stable category $\underline{\Gproj} A$ is a triangulated
category.

An Artin algebra $A$ is {\it CM-free} if $\Gproj A = \proj A$,
and it is {\it finite Cohen-Macaulay type (or CM-finite)} if there are only finitely many
indecomposable objects in $\Gproj A$, up to
isomorphism \cite[Example 8.4 (2)]{Bel05}. In the latter case,
let $M$ be an additive generator of $\Gproj A$, that is, $\add(M) =
\Gproj A$. Then $\Gamma _A:= \End _A(M) = \Hom _A (M,M)$ is called the {\it Cohen-Macaulay
Auslander algebra} of $A$ \cite{Bel05, Bel11}.

\begin{definition}{\rm (Beilinson-Bernstein-Deligne \cite{BBD82})
Let $\mathcal{T}_1$, $\mathcal{T}$ and $\mathcal{T}_2$ be
triangulated categories. A {\it recollement} of $\mathcal{T}$
relative to $\mathcal{T}_1$ and $\mathcal{T}_2$ is given by
$$\xymatrix@!=4pc{ \mathcal{T}_1 \ar[r]^{i_*=i_!} & \mathcal{T} \ar@<-3ex>[l]_{i^*}
\ar@<+3ex>[l]_{i^!} \ar[r]^{j^!=j^*} & \mathcal{T}_2
\ar@<-3ex>[l]_{j_!} \ar@<+3ex>[l]_{j_*}}\ \ \ \ \ \ \ \ \ \ \ \ (R)$$
and denoted by $(\mathcal{T}_1, \mathcal{T}, \mathcal{T}_2,
i^*,i_*=i_!, i^!, j_!,j^!=j^*, j_* )$ (or just $(\mathcal{T}_1, \mathcal{T}, \mathcal{T}_2)$)
such that

(R1) $(i^*,i_*), (i_!,i^!), (j_!,j^!)$ and $(j^*,j_*)$ are adjoint
pairs of triangle functors;

(R2) $i_*$, $j_!$ and $j_*$ are full embeddings;

(R3) $j^!i_*=0$ (and thus also $i^!j_*=0$ and $i^*j_!=0$);

(R4) for each $X \in \mathcal {T}$, there are triangles

$$\begin{array}{l} j_!j^!X \rightarrow X  \rightarrow i_*i^*X  \rightarrow
\\ i_!i^!X \rightarrow X  \rightarrow j_*j^*X  \rightarrow
\end{array},$$ where the arrows to and from $X$ are the counits and the
units of the adjoint pairs respectively.
}\end{definition}

An {\it upper (resp. lower) recollement} of $\mathcal{T}$
relative to $\mathcal{T}_1$ and $\mathcal{T}_2$ is the upper (resp. lower) two rows of (R)
such that $i^*,i_*,j_!$ and $j^!$ satisfy the above conditions \cite{BGS88,Koe91,Par89}.
The upper (resp. lower) recollement are
also called left (resp. right) recollement in some literature.

\begin{definition}{\rm
Let $\mathcal{T}_1$, $\mathcal{T}$ and $\mathcal{T}_2$ be
triangulated categories, and $n$ a positive integer. An {\it
$n$-recollement} of $\mathcal{T}$ relative to $\mathcal{T}_1$ and
$\mathcal{T}_2$ is given by $n+2$ layers of triangle functors
$$\xymatrix@!=4pc{ \mathcal{T}_1 \ar@<+1ex>[r]|{i_2} \ar@<-3ex>[r]|{i_4}_\vdots & \mathcal{T}
\ar@<+1ex>[r]|{j_2} \ar@<-3ex>[r]|{j_4}_\vdots \ar@<-3ex>[l]|{i_1}
\ar@<+1ex>[l]|{i_3} & \mathcal{T}_2 \ar@<-3ex>[l]|{j_1}
\ar@<+1ex>[l]|{j_3}}$$ such that every consecutive three layers form
a recollement, and denoted by $(\mathcal{T}_1, \mathcal{T},
\mathcal{T}_2, i_1,i_2, \cdots ,i_{n+2}, j_1, j_2, \cdots ,
j_{n+2})$.
This definition is taken from ladder \cite{BGS88, AKLY17} with a minor modification.}\end{definition}

\section{\large Recollements and Cohen-Macaulay Auslander algebras}
\indent\indent In this section, we show that a $4$-recollement of unbounded derived categories
of algebras induces a lower recollement of $\underline{\Gproj }$-level,
and then, we prove Theorem A. Let's begin with some definition in \cite{HP17}.

Suppose that $A$ and $B$ are Artin algebras. Following \cite{HP17},
a triangle functor $F: D^b(\mod A)\rightarrow D^b(\mod B)$ is said to be
{\it nonnegative} if $F$ satisfies the following conditions:
(1) $F(X)$ is isomorphic to a complex with zero homology in all
negative degrees, for all $X \in \mod A$.
(2) $F(A)$ is isomorphic to a complex in $K^b(\proj B)$
with zero terms in all negative degrees.
According to \cite{HP17}, a nonnegative triangle functor $F: D^b(\mod A)\rightarrow D^b(\mod B)$
induces a stable functor $\overline{F}: \underline{\mod A} \rightarrow \underline{\mod B}$.
Indeed, for any $X \in \mod A$, assume $F(X)$ has a projective resolution of
the form $$(P^\bullet,d) : \cdots \longrightarrow  P^{-1} \longrightarrow
P^{0} \longrightarrow P^{1} \longrightarrow
\cdots \longrightarrow P^{m} \longrightarrow 0.$$
Then $\overline{F}(X)$ is defined to be the cokernel of $d^{-1}$.

\begin{lemma}\label{lemme-nonnegative}
Let $F : \mathcal{D}A \rightarrow \mathcal{D}B$ be a triangle functor
which has a left adjoint $L : \mathcal{D}B \rightarrow \mathcal{D}A$.
If $L$ and $F$ restrict to
$ K^b(\proj )$, then $F$ restrict to
$ D^b(\mod)$, and $F[-n]$ is nonnegative for some positive integer $n$.
\end{lemma}
\begin{proof}
It follows from \cite[Lemma 2.7]{AKLY17} that $F$ restrict to
$ D^b(\mod)$. Since $L(B)\in K^b(\proj A)$, there is a positive integer $r$
such that $L(B)$ is quasi-isomorphic to a complex $Q^ \bullet$ in $K^b(\proj A)$ with $Q^i=0$,
for any $i>r$. Similarly, there is a positive integer $s$
such that $F(A)$ is quasi-isomorphic to a complex $P^ \bullet$ in $K^b(\proj B)$ with $P^i=0$,
for any $i<-s$. Take $n=\max \{ r,s\}$, and then $L[n](B)$ is isomorphic to a
complex in $K^b(\proj A)$ with zero terms in all positive degrees,
and $F[-n](A)$ is isomorphic to a
complex in $K^b(\proj B)$ with zero terms in all negative degrees.
Moreover, for any $X \in \mod A$,
$H^i(F[-n]X)=\Hom_{\mathcal{D}B}(B,F[-n]X[i])\cong
\Hom_{\mathcal{D}A}(L[n]B,X[i])=0$, for any $i<0$.
Therefore, $F[-n]$ is nonnegative.

\end{proof}

Following \cite{Kato98}, we denote $D^b(A)_{fGd}$
the full subcategory of $D^b(A)$ consisting of objects
$X^\bullet$ such that $X^\bullet$ is quasi-isomorphic to
a complex in $K^b(\Gproj A)$.

\begin{proposition}\label{proposition-4-recoll}
Let $A$, $B$ and $C$ be Artin $R$-algebras, and
$$\xymatrix@!=6pc{ \mathcal{D}B \ar@<2.4ex>[r]|{i_*} \ar@<-4ex>[r]
\ar@<-0.8ex>[r]|{i_?}  & \mathcal{D}A \ar@<-4ex>[l]|{i^*}
\ar@<-0.8ex>[l]|{i^!} \ar@<+2.4ex>[l] \ar@<-0.8ex>[r]|{j^?}
\ar@<2.4ex>[r]|{j^*}  \ar@<-4ex>[r] & \mathcal{D}C \ar@<-4ex>[l]|{j_!}
\ar@<-0.8ex>[l]|{j_*} \ar@<2.4ex>[l]}$$ be a $4$-recollement.
Then there are two induced lower recollements
$$\xymatrix@!=6pc{D^b(B)_{fGd}
\ar@<1ex>[r]|{i_*} & D^b(A)_{fGd} \ar@<1ex>[l]|{i^!}
\ar@<1ex>[r]|{j^*} & D^b(C)_{fGd} \ar@<1ex>[l]|{j_*}
 } $$ and
 $$\xymatrix@!=6pc{\underline{\Gproj }B
\ar@<1ex>[r]|{\overline {i_*}}& \underline{\Gproj }A \ar@<1ex>[l]|{\overline{i^!}}
\ar@<1ex>[r]|{\overline{j^*}}  & \underline{\Gproj }C \ar@<1ex>[l]|{\overline{j_*}}
 } .$$
\end{proposition}
\begin{proof}
According to \cite[Lemma 2.9]{AKLY17},
the functors $i_*, i^!, i_?, j^*, j_*, j^?$ restrict to both
$ K^b(\proj )$ and $\mathcal{D}^b(\mod)$, $i^*$ and $j_!$ restrict to
$ K^b(\proj )$.
Due to Lemma~\ref{lemme-nonnegative}, we may assume both $i_*$ and $j_*$ are nonnegative.
Then it follows from \cite[Theorem 5.3]{HP17} that $i_*$ (resp. $j_*$)
sends modules in $\Gproj B$ (resp. $\Gproj C$) to complexes in $K^b(\Gproj A)$, up to quasi-isomorphism.
That is, both $i_*$ and $j_*$ restrict to $D^b(-)_{fGd}$.
By Lemma~\ref{lemme-nonnegative}, there
exist two integers $n, m > 0$ such that $i^![-n]$ and $j^*[-m]$ are nonnegative.
Therefore, both $i^![-n]$ and $j^*[-m]$ restrict to $D^b(-)_{fGd}$, and
so do $i^!$ and $j^*$. Hence, we obtain the desired lower recollement of $D^b(-)_{fGd}$-level.

From \cite[Proposition 3.4]{Kato98}
or \cite[Theorem 4]{PZ15}, there is a triangle equivalence
$\underline{\Gproj }A \cong D^b(A)_{fGd}/K^b(\proj A)$.
Now the lower recollement of $\underline{\Gproj }$-level exists,
according to \cite[Lemma 2.3]{Lu17}.
\end{proof}

\begin{remark}
{\rm The lower recollement of $\underline{\Gproj }$-level
is obtained in \cite[Theorem 3.3]{Zhang13}
for triangular matrix algebra.}

\end{remark}

In the following, we denote by ${X^\bullet}^ \perp$
the full subcategory of $\mathcal{D}A$ consisting of all objects
$Y^\bullet \in \mathcal{D}A$ such that $\Hom _{\mathcal{D}A}(X^\bullet,Y^\bullet[n])=0$,
for all $n\in \mathbb{Z}$.
The following proposition
is essentially due to K\"{o}nig \cite{Koe91}.
\begin{proposition}\label{proposition-recoll}
Let $A$, $B$ and $C$ be Artin $R$-algebras.
Then $\mathcal{D}A$ admits a $2$-recollement relative to
$\mathcal{D}B$ and $\mathcal{D}C$ if and only if there are objects
$T^\bullet$ and $U^\bullet$ in $K^b(\proj A)$ such that

{\rm (1)} $\End _{\mathcal{D}A}(T^\bullet)\cong C$ and $\End _{\mathcal{D}A}(U^\bullet)\cong B$
as algebras;

{\rm (2)} $\Hom _{K^b(\proj A)}
(T^\bullet, T^\bullet [i])\cong \Hom _{K^b(\proj A)}
(U^\bullet, U^\bullet [i])\cong 0$, for any $i\neq 0$;

{\rm (3)} $U^\bullet \in {T^\bullet}^ \perp$;

{\rm (4)} ${T^\bullet}^ \perp \bigcap {U^\bullet}^ \perp =\{0\}$.
\end{proposition}
\begin{proof}
Assume there are objects
$T^\bullet$ and $U^\bullet$ with the desired properties.
From \cite[Theorem 1]{Koe91}, there is a recollement
$(\mathcal{D}^-(\Mod B), \mathcal{D}^-(\Mod A), \mathcal{D}^-(\Mod C),
\linebreak
i^*,i_*, i^!, j_!,j^*, j_* )$ with $j_!C\cong T^\bullet$
and $i_*B\cong U^\bullet$.
By \cite[Proposition 4.1]{AKLY17}, this $\mathcal{D}^-(\Mod-)$
recollement can be lifted to a recollement
of $\mathcal{D}(\Mod)$-level. Now $U^\bullet\in K^b(\proj A)$
implies that this recollement can be extended to a $2$-recollement
(ref. \cite[Lemma 3.1]{AKLY17}).
The converse is clear.
\end{proof}

\begin{theorem}\label{theorem-1}
Let $A$, $B$ and $C$ be Artin $R$-algebras
such that
$B$ and $C$ are CM-finite. Assume there is a $4$-recollement
$$\xymatrix@!=6pc{ \mathcal{D}B \ar@<2.4ex>[r]|{i_*} \ar@<-4ex>[r]
\ar@<-0.8ex>[r]|{i_?}  & \mathcal{D}A \ar@<-4ex>[l]|{i^*}
\ar@<-0.8ex>[l]|{i^!} \ar@<+2.4ex>[l] \ar@<-0.8ex>[r]|{j^?}
\ar@<2.4ex>[r]|{j^*}  \ar@<-4ex>[r] & \mathcal{D}C \ar@<-4ex>[l]|{j_!}
\ar@<-0.8ex>[l]|{j_*} \ar@<2.4ex>[l]}.$$
Then there is a $2$-recollement
$$\xymatrix@!=6pc{ \mathcal{D}\Gamma _C \ar@<-2.4ex>[r]
\ar@<0.8ex>[r]  & \mathcal{D}H \ar@<-2.4ex>[l]
\ar@<+0.8ex>[l] \ar@<-2.4ex>[r] \ar@<+0.8ex>[r]
& \mathcal{D}\Gamma _B \ar@<-2.4ex>[l]
\ar@<+0.8ex>[l]},$$
where $\Gamma _B$ and $\Gamma _C$ are the corresponding Cohen-Macaulay Auslander algebras
of $B$ and $C$, $H=\End _A(A\oplus \overline {i_*}X\oplus \overline {j_*}Y)$,
and $X$ (resp. $Y$) is the direct sum of all the indecomposable
non-projective objects in $\Gproj B$ (resp. $\Gproj C$).

\end{theorem}

\begin{proof}
Without loss of generality, we may assume both $i_*$ and $j_*$
are nonnegative, and then $i_*B$(resp. $j_*C$) is quasi-isomorphic
to a projective complex $P^\bullet$(resp. $Q^\bullet$)
of the following form:
 $$P^\bullet : 0 \longrightarrow  P^{0} \longrightarrow
P^{1} \longrightarrow
\cdots \longrightarrow P^{n} \longrightarrow 0,$$
 $$Q^\bullet : 0 \longrightarrow  Q^{0} \longrightarrow
Q^{1} \longrightarrow
\cdots \longrightarrow Q^{m} \longrightarrow 0.$$
Moreover, $i_*X$(resp. $j_*Y$) is quasi-isomorphic
to a complex $P_X^\bullet$(resp. $Q_Y^\bullet$)
of the following form:
 $$P_X^\bullet : 0 \longrightarrow  \overline{X} \longrightarrow
P_X^{1} \longrightarrow
\cdots \longrightarrow P_X^{n} \longrightarrow 0,$$
 $$Q_Y^\bullet : 0 \longrightarrow  \overline{Y} \longrightarrow
Q_Y^{1} \longrightarrow
\cdots \longrightarrow Q_Y^{m} \longrightarrow 0,$$
where $\overline{X}=\overline {i_*}X\in \Gproj A$, $\overline{Y}=\overline {j_*}Y\in\Gproj A$,
and $P_X^{i}$, $Q_Y^{i}\in \proj A$, for any $i\geq 1$.
Set $N=A\oplus \overline{X}\oplus \overline{Y}$ and $H=\End _AN$. Then
the functor $\Hom_A(N,-)$ gives an equivalence $\add N\cong \proj H$,
and it induces an equivalence $K^b(\add N)\cong K^b(\proj H)$.
Clearly, $P^\bullet, P_X^\bullet, Q^\bullet$ and $Q_Y^\bullet\in
K^b(\add N)$. Therefore, $\Hom _A(N,P^\bullet \oplus P_X^\bullet)\in K^b(\proj H)$
and $\Hom _A(N,P^\bullet \oplus P_X^\bullet)\in K^b(\proj H)$.
Set $T^\bullet:=\Hom _A(N,P^\bullet \oplus P_X^\bullet)$
and $U^\bullet:=\Hom _A(N,Q^\bullet \oplus Q_Y^\bullet)$.
Now we claim the $H$-module complexes $T^\bullet$
and $U^\bullet$ satisfy
all conditions in Proposition~\ref{proposition-recoll}.

{\it Step 1} \ Due to the equivalence $K^b(\add N)\cong K^b(\proj H)$,
we have $$\End _{\mathcal{D}H}(T^\bullet)\cong \End _{K^b(\proj H)}(T^\bullet)\cong
\End _{K^b(\add N)}(P^\bullet \oplus P_X^\bullet)=
\End _{K^b(A)}(P^\bullet \oplus P_X^\bullet),$$
and by \cite[Lemma 2.2]{HX10}, we have $\End _{K^b(A)}(P^\bullet \oplus P_X^\bullet)\cong
\End _{\mathcal{D}A}(P^\bullet \oplus P_X^\bullet).$
Therefore, $$\End _{\mathcal{D}H}(T^\bullet)\cong
\End _{\mathcal{D}A}(P^\bullet \oplus P_X^\bullet)\cong
\End _{\mathcal{D}B}(B \oplus X)\cong \Gamma _B.$$
Similarly, we have $$\End _{\mathcal{D}H}(U^\bullet)\cong
\End _{\mathcal{D}A}(Q^\bullet \oplus Q_Y^\bullet)\cong
\End _{\mathcal{D}C}(C \oplus Y)\cong \Gamma _C.$$

{\it Step 2} \ Only need to show $\Hom _{K^b(A)}
(P^\bullet \oplus P_X^\bullet, P^\bullet \oplus P_X^\bullet [i])
\cong \Hom _{K^b(A)}
(Q^\bullet \oplus Q_Y^\bullet, Q^\bullet \oplus Q_Y^\bullet [i])\cong 0$,
for any $i\neq 0$, and this is exactly the same as
\cite[Lemma 3.7]{Pan12}, so we skip the
proof.

{\it Step 3} \ Only need to show $\Hom _{K^b(A)}
(P^\bullet \oplus P_X^\bullet, Q^\bullet \oplus Q_Y^\bullet [i])
\cong 0$,
for any $i\in \mathbb{Z}$. Indeed, if $i\neq 0$ then
we get $\Hom _{K^b(A)}
(P^\bullet \oplus P_X^\bullet, Q^\bullet \oplus Q_Y^\bullet [i])
\cong 0$ using the same argument as \cite[Lemma 3.7]{Pan12}.
Moreover, by \cite[Lemma 2.2]{HX10}, we have
$\Hom _{K^b(A)}
(P^\bullet \oplus P_X^\bullet, C \oplus Q_Y^\bullet )
\cong \Hom _{\mathcal{D}A}
(P^\bullet \oplus P_X^\bullet, Q^\bullet \oplus Q_Y^\bullet )
\cong \Hom _{\mathcal{D}A}
(i_*B \oplus i_*X, j_*C \oplus j_*Y )\cong0.$

{\it Step 4} \ From the triangle
$i_*i^!A\rightarrow A \rightarrow j_*j^*A\rightarrow $,
we have that $\add (P^\bullet \oplus Q^\bullet)$ generates
$K^b(\proj A)$ as a triangulated category.
Then, by the triangles $\tau _{\geq 1}(P_X^\bullet)
\rightarrow P_X^\bullet \rightarrow \overline{X}\rightarrow$
and $\tau _{\geq 1}(Q_Y^\bullet)
\rightarrow Q_Y^\bullet \rightarrow \overline{Y}\rightarrow$,
we get that both $\overline{X}$ and $\overline{Y}$
are in the triangulated subcategory generated by $\add (P^\bullet \oplus
P_X^\bullet\oplus Q^\bullet \oplus Q_Y^\bullet )$, that is,
$\add (P^\bullet \oplus
P_X^\bullet\oplus Q^\bullet \oplus Q_Y^\bullet )$
generates $K^b(\add N)$ as
a triangulated category. Applying the functor $\Hom_A(N,-)$,
we get that $\add (T^\bullet\oplus U^\bullet)$
generates $K^b(\proj H)$ as
a triangulated category, and thus,
${T^\bullet}^ \perp \bigcap {U^\bullet}^ \perp =\{0\}$.

\end{proof}

Applying Theorem 1 to derived equivalence, we get the following corollary,
which is a generalization of the main theorem in \cite{Pan12, PZ15}.
\begin{corollary}
Let $A$ and $B$ be Artin $R$-algebras of finite Cohen-Macaulay type.
If $A$ and $B$ are derived equivalent, then their Cohen-Macaulay Auslander
algebras $\Gamma _A$ and $\Gamma _B$ are derived equivalent.
\end{corollary}
\begin{proof}
Note that any derived equivalent can be viewed as a $4$-recollement whose
right term is zero. So, by Theorem~\ref{theorem-1},
we get an equivalence
$\mathcal{D}H\cong \mathcal{D}\Gamma _B$, where
$H=\End _A(A\oplus \overline {i_*}X)$,
and $X$ is the direct sum of all the indecomposable
non-projective objects in $\Gproj B$.
By Proposition~\ref{proposition-4-recoll},
the functor $\overline {i_*}$ induces an equivalence $\underline{\Gproj }B\cong \underline{\Gproj }A$,
and thus, $\add(A\oplus \overline {i_*}X) = \Gproj A$
and hence the proof is complete.
\end{proof}

\section{\large Recollements and homological conjectures}
\indent\indent The following conjecture,
proposed by Auslander and Reiten \cite{AR75}, plays an important role in the
understanding of the famous Nakayama
conjecture.

{\bf Auslander-Reiten conjecture.}
A finitely generated
module $M$ over an Artin algebra $A$ is projective if
$\Ext ^i _A(M, M\oplus A) = 0$, for any $i \geq 1$.

The Auslander-Reiten Conjecture is proved for several classes of algebras,
such as algebras of finite representation type \cite{AR75}
and symmetric algebras with radical cube zero \cite{Hos84}.
However, it remains open in general.
As a special case of Auslander-Reiten conjecture, Luo and Huang \cite{LH08}
proposed the following:

{\bf Gorenstein projective conjecture.} A finitely generated Gorenstein projective
module $M$ over an Artin algebra $A$ is projective if
 $\Ext ^i _A(M,M)
\linebreak
= 0$, for any $i \geq 1$.

The Auslander-Reiten conjecture
and Gorenstein projective conjecture coincide when $A$ is a
Gorenstein algebra, but it seems not true in general.
Moreover, the
Gorenstein projective conjecture is proved for
CM-finite algebras \cite{Zhang12}. For more development of this conjecture we refer to
\cite{LJ16}.

In this section, we will investigate
the invariance of Gorenstein projective conjecture and
Auslander-Reiten conjecture along recollement
of derived categories.

For an Artin algebra $A$, set $^\bot A := \{M \in  \mod A | \Ext ^i(M, A) = 0, \forall i > 0 \}$,
and denote $\underline{^\bot A}$ the stable category of $^\bot A$
modulo finitely generated projective $A$-modules. According to \cite[Theorem 2.12]{BM94}, $\underline{^\bot A}$
is a left triangulated
category with the standard left triangulated structure.

\begin{lemma}\label{lemma-3-recoll}
Let $A$, $B$ and $C$ be Artin $R$-algebras, and
$$\xymatrix@!=6pc{ \mathcal{D}B \ar@<-1.5ex>[r]|{i_?}\ar@<+1.5ex>[r]|{i_*} &\mathcal{D}A \ar[l]|{i^!}
\ar@<+3ex>[l] \ar@<-3ex>[l] \ar@<-1.5ex>[r]|{j^?}\ar@<+1.5ex>[r]|{j^*} &
\mathcal{D}C \ar[l]|{j_*}
\ar@<+3ex>[l] \ar@<-3ex>[l]}$$ be a $3$-recollement such that $i_*$
is nonnegative. Then there are two fully faithful
functors $\overline{i_*}: \underline{\Gproj }B \rightarrow \underline{\Gproj }A$
and
$\overline{i_*}: \underline{^\bot B} \rightarrow \underline{^\bot A}$.
\end{lemma}
\begin{proof}
From \cite[Proposition 4.8 and Proposition 5.2]{HP17}, we have
the following
communicative diagrams
$$\xymatrix{\underline{\Gproj }B  \ar[r]^{\overline{i_*}}\ar @{^{(}->}[d] & \underline{\Gproj }A
\ar @{^{(}->}[d]
\\D_{sg}(B) \ar[r]^{i_*}  & D_{sg}(A)}\begin{array}{c}
\\  \\ \mbox{and} \\  \end{array}
\xymatrix{\underline{^\bot B}  \ar[r]^{\overline{i_*}}\ar @{^{(}->}[d] & \underline{^\bot A}
\ar @{^{(}->}[d]
\\D_{sg}(B) \ar[r]^{i_*}  & D_{sg}(A).}$$
Here, the canonical embedding $\underline{\Gproj }A \hookrightarrow D_{sg}(A)$
is clear, and the fully faithfulness of $\underline{^\bot A} \hookrightarrow D_{sg}(A)$
can be proved using the same argument as \cite[Theorem 2.1]{Rick89}.

According to \cite[Lemma 2.9]{AKLY17},
the functors $i_*, i^!, j^*$ and $j_*$ restrict to both
$ K^b(\proj )$ and $\mathcal{D}^b(\mod)$.
Then, it follows from \cite[Lemma 2.3]{Lu17} that
$(D_{sg}(B),D_{sg}(A),D_{sg}(C), i_*, i^!, j^*, j_*)$
is a lower recollement. Therefore, the functor
$i_* : D_{sg}(B)\rightarrow D_{sg}(A)$ is fully faithful.
Now the proof is finished, just using the above
communicative diagrams.

\end{proof}

\begin{theorem}\label{theorem-3-GPC-ARC}
Let $A$, $B$ and $C$ be Artin $R$-algebras, and $\mathcal{D}A$ admit a $3$-recollement relative
to $\mathcal{D}B$ and $\mathcal{D}C$. If
$A$ satisfies the Gorenstein projective
conjecture (resp. Auslander-Reiten conjecture), then so does $B$.
\end{theorem}
\begin{proof}
Assume $A$ satisfy the Gorenstein projective
conjecture. Applying shift functors, we may assume the functor $i_*$
in the $3$-recollement is nonnegative,
as in Lemma~\ref{lemma-3-recoll}. Then the functor $\overline{i_*} :
\underline{\Gproj }B \rightarrow \underline{\Gproj }A$ is fully faithful.
Let $X \in \Gproj B$ with $\Ext _B^i(X, X)=0$, for any $i>0$.
Then $$\Ext _A^i(\overline{i_*}X, \overline{i_*}X) \cong
\underline{\Hom} _A(\Omega ^i_A\overline{i_*}X, \overline{i_*}X)=
\Hom _{\underline{\Gproj }A}(\Omega ^i_A\overline{i_*}X, \overline{i_*}X),$$
where the first isomorphism is proved in \cite[Lemma 3.3]{Zhang12}.
On the other hand, it follows from \cite[Corollary 4.12]{HP17} that
$\Omega ^i_A\overline{i_*}\cong \overline{i_*}\Omega ^i_B$. There, the above
$\Hom$-space is isomorphic to
$$\Hom _{\underline{\Gproj }A}(\overline{i_*}\Omega ^i_BX, \overline{i_*}X)\cong
\Hom _{\underline{\Gproj }B}(\Omega ^i_BX, X)\cong \Ext _B^i(X, X).$$
As a result, we get $\Ext _A^i(\overline{i_*}X, \overline{i_*}X) \cong 0$,
for any $i>0$. So, by assumption, we see that the $A$-module $\overline{i_*}X$
is projective, and by the definition of $\overline{i_*}$, we obtain that
$i_*X\in K^b(\proj A)$. Therefore, $X\cong i^*i_*X \in K^b(\proj B)$, and this
is equivalent to $\pd X<\infty$. Now note that $X \in \Gproj B$, and thus, $X$ is a projective $B$-module.

The statement on Auslander-Reiten conjecture can be proved similarly,
just replacing $\Gproj B$ (resp. $\Gproj A$) with $^\bot B$ (resp. $^\bot A $).

\end{proof}

\begin{corollary}\label{corollary-4-ARC-GPC}
Let $A$, $B$ and $C$ be Artin $R$-algebras, and $\mathcal{D}A$ admit a $4$-recollement relative
to $\mathcal{D}B$ and $\mathcal{D}C$. If
$A$ satisfies the Gorenstein projective
conjecture (resp. Auslander-Reiten conjecture), then so do $B$ and $C$.
\end{corollary}
\begin{proof}
Let $(\mathcal{D}B, \mathcal{D}A,
\mathcal{D}C, i_1,i_2, \cdots ,i_{6}, j_1, j_2, \cdots ,
j_{6})$ be a $4$-recollement. Then both $(\mathcal{D}B, \mathcal{D}A,
\mathcal{D}C, i_1,i_2, \cdots ,i_{5}, j_1, j_2, \cdots ,
j_{5})$ and $(\mathcal{D}C, \mathcal{D}A,
\mathcal{D}B, j_2, j_3, \cdots ,
\linebreak
j_{6}, i_2,i_3, \cdots ,i_{6}, )$
are $3$-recollements.
Therefore, the statement follows
from Theorem~\ref{theorem-3-GPC-ARC}.
\end{proof}

In \cite{PZ15}, Pan and Zhang proved that Gorenstein projective conjecture is
invariant under standard derived equivalence, and they asked
whether it is also true for general derived equivalence. Now
we will give a positive answer by applying Theorem~\ref{theorem-3-GPC-ARC}
to derived equivalence.

\begin{corollary}
Let $A$ and $B$ be Artin $R$-algebras which are derived equivalent.
Then $A$ satisfy the Gorenstein projective conjecture if and
only if so does $B$.
\end{corollary}
\begin{proof}
Since there are two $3$-recollements
$(\mathcal{D}B, \mathcal{D}A,0)$ and $(\mathcal{D}A, \mathcal{D}B,0)$,
this corollary follows from Theorem~\ref{theorem-3-GPC-ARC}.
\end{proof}

Following \cite{AHV16}, we denote $K^{-,bb}(\proj A)$ the full subcategory
of $K^{-}(\proj A)$ consisting of complexes
$X^\bullet$ such that there is an integer $n = n(X^\bullet)$, in which $H^i (X) = 0$, for $i \leq n$ and
$H^i (\Hom _A(X^\bullet,A)) = 0$, for $i \geq -n$. By \cite[Proposition
5.1.7]{AHV16}, there exists an equivalence
$\underline{^\bot A} \cong K^{-,bb}(\proj A)/K^b(\proj A)$
of left triangulated categories.

\begin{lemma}\label{lemma-kbb}
The category $K^{-,bb}(\proj A)$ is the full subcategory of $D^b(\mod A)$
consisting of those objects $X^\bullet$ with the following property: for any $P^\bullet \in K^b(\proj A)$,
there exists an integer $n$ such that $\Hom _{DA}(X^\bullet, P^\bullet [i])=0$,
for any $i> n$.
\end{lemma}
\begin{proof}
Assume $X^\bullet \in K^{-,bb}(\proj A)$. Then there exists an integer $n$ such that
$\Hom _{DA}(X^\bullet, A [i])=0$, for any $i> n$. Doing induction
on the length
of the complex $P^\bullet$, we get $\Hom _{DA}(X^\bullet, P^\bullet [i])=0$,
for sufficient large $i$. Conversely, if $X^\bullet$ has the desired property,
then we obtain $X^\bullet \in K^{-,bb}(\proj A)$ by taking $P^\bullet =A$.
\end{proof}

\begin{proposition}\label{proposition-4-recoll-left-orth}
Let $A$, $B$ and $C$ be Artin $R$-algebras, and
$$\xymatrix@!=6pc{ \mathcal{D}B \ar@<2.4ex>[r]|{i_*} \ar@<-4ex>[r]
\ar@<-0.8ex>[r]|{i_?}  & \mathcal{D}A \ar@<-4ex>[l]
\ar@<-0.8ex>[l]|{i^!} \ar@<+2.4ex>[l] \ar@<-0.8ex>[r]|{j^?}
\ar@<2.4ex>[r]|{j^*}  \ar@<-4ex>[r] & \mathcal{D}C \ar@<-4ex>[l]
\ar@<-0.8ex>[l]|{j_*} \ar@<2.4ex>[l]}$$ be a $4$-recollement.
Then there is an induced lower recollement of left triangulated categories
 $$\xymatrix@!=6pc{\underline{^\bot B}
\ar@<1ex>[r]|{\widetilde {i_*}}& \underline{^\bot A} \ar@<1ex>[l]|{\widetilde{i^!}}
\ar@<1ex>[r]|{\widetilde{j^*}}  & \underline{^\bot C} \ar@<1ex>[l]|{\widetilde{j_*}}
 } .$$
\end{proposition}
\begin{proof}
According to \cite[Lemma 2.9]{AKLY17},
the functors $i_*, j^*, i^!, j_*$ restrict to $\mathcal{D}^b(\mod)$
and $ K^b(\proj )$,
and the functors $i_?$ and $j^?$ restrict
to $ K^b(\proj )$. By Lemma~\ref{lemma-kbb}, the functors
$i_*, j^*, i^!$ and $j_*$ restrict
to $ K^{-,bb}(\proj )$. Thus, by \cite[Lemma 2.3]{Lu17}, there is a lower recollement
 $$\xymatrix@!=6pc{\frac {K^{-,bb}(\proj B)}{K^b(\proj B)}
\ar@<1ex>[r]|{\widetilde {i_*}}& \frac {K^{-,bb}(\proj A)}{K^b(\proj A)} \ar@<1ex>[l]|{\widetilde{i^!}}
\ar@<1ex>[r]|{\widetilde{j^*}}  & \frac {K^{-,bb}(\proj C)}{K^b(\proj C)} \ar@<1ex>[l]|{\widetilde{j_*}}
 }.$$
 Now, the equivalence
$\underline{^\bot A} \cong K^{-,bb}(\proj A)/K^b(\proj A)$ implies the desired lower recollement.

\end{proof}

\begin{remark}
{\rm A special case of Proposition~\ref{proposition-4-recoll-left-orth}
is considered in \cite[Proposition 5.1.11]{AHV16}.}

\end{remark}

\begin{theorem}\label{theorem-3}
Let $A$, $B$ and $C$ be Artin $R$-algebras,
and $\mathcal{D}A$ admit a $3$-recollement relative
to $\mathcal{D}B$ and $\mathcal{D}C$. Assume $\gl C<\infty$.
Then $A$ satisfies Gorenstein projective
conjecture (resp. Auslander-Reiten conjecture)
if and only if so does $B$.
\end{theorem}
\begin{proof}
Let $(\mathcal{D}B, \mathcal{D}A,
\mathcal{D}C, i_1,i_2, \cdots ,i_{5}, j_1, j_2, \cdots ,
j_{5})$ be a $3$-recollement. Then the assumption $\gl C<\infty$
implies that $j_4(A)\in K^b(\proj C)$, that is,
the $3$-recollement can be extended one step downward.
By Proposition~\ref{proposition-4-recoll} and Proposition~\ref{proposition-4-recoll-left-orth},
there are two lower recollements $(\underline{\Gproj }B,\underline{\Gproj }A,\underline{\Gproj }C)$
and $(\underline{^\bot B}, \underline{^\bot A}, \underline{^\bot C})$.
On the other hand, the contidion $\gl C<\infty$
implies $\underline{\Gproj }C=0$ and $\underline{^\bot C}=0$.
Therefore, there are (left) triangle equivalence
$\underline{\Gproj }B\cong \underline{\Gproj }A$ and $\underline{^\bot B}\cong
\underline{^\bot A}$. Now, the theorem can be proved
by the same argument as Theorem~\ref{theorem-3-GPC-ARC}.

\end{proof}

\begin{theorem}\label{theorem-4}
Let $A$, $B$ and $C$ be Artin $R$-algebras,
and $\mathcal{D}A$ admit a $4$-recollement relative
to $\mathcal{D}B$ and $\mathcal{D}C$. Then the following hold true:

{\rm (1)} If one of $B$ and $C$
is CM-free, and the other satisfies Gorenstein projective
conjecture, then $A$
satisfies Gorenstein projective
conjecture;

{\rm (2)} If $\underline{^\bot C}=0$ and
$B$ satisfies Auslander-Reiten
conjecture, then $A$
satisfies Auslander-Reiten
conjecture, and the statement is also
true if we exchange the roles of $B$ and $C$.

\end{theorem}
\begin{proof}
By Proposition~\ref{proposition-4-recoll} and Proposition~\ref{proposition-4-recoll-left-orth},
there exist two lower recollements $(\underline{\Gproj }B,\underline{\Gproj }A,\underline{\Gproj }C)$
and $(\underline{^\bot B}, \underline{^\bot A}, \underline{^\bot C})$.
Now, the theorem can be proved
by the same argument as Theorem~\ref{theorem-3}.
\end{proof}

\section{\large Examples}
\indent\indent  In this section, we give two
examples to demonstrate how our
results can be used. The first one is a $4$-recollement which can not be extended to
a $5$-recollement, and the second is reducing Auslander-Reiten conjecture
by recollement.

\begin{example}
{\rm Let $k$ be a field and $Q$ be the following quiver
$$\xymatrix{4 \ar@<-0.5ex>[r]_{\beta _1} & 5
\ar@<-0.5ex>[l]_{\alpha _1} \ar@<-0.5ex>[r]_{\beta _2}
& 6 \ar[r]_{\beta _3} \ar@<-0.5ex>[l]_{\alpha _2}
& 7  \\  1 \ar[u]_{\gamma _1} \ar@<-0.5ex>[r]_{\delta _1} & 2
\ar@<-0.5ex>[l]_{\theta _1} \ar[u]_{\gamma _2} & 3 \ar[l]_{\theta _2}}.$$
Let $A=kQ/I$ where $I$ is the ideal of $kQ$ generated by
$\alpha _1 \beta _1$, $\beta _1
\alpha _1$, $\alpha _2 \beta _2$, $\beta _2
\alpha _2$, $\beta _2 \beta _3$,
$\theta _1 \delta _1$, $\delta _1
\theta _1$, $\theta _2 \theta _1$, $\theta _2 \gamma _2$,
$\gamma _1 \beta _1- \delta _1 \gamma _2 $ and $\theta _1 \gamma _1
-\gamma _2\alpha _1$. This algebra $A$ is studied in \cite[Example 3.16]{Lu17}.
Indeed, $A = \left[\begin{array}{cc} B & 0
\\ M & C \end{array}\right] $ is a triangular matrix algebra,
where $B=(1-e)A(1-e)$, $C=eAe$, $e=e_1+e_2+e_3$.
Obviously, $\pd  _CM=1$ and $M_B$ is projective.
By \cite[Example 3.4]{AKLY17},
$\mathcal{D}A$ admits a $4$-recollement relative
to $\mathcal{D}C$ and $\mathcal{D}B$.
Now we claim that this $4$-recollement can not be extended to
a $5$-recollement. Otherwise, assume there is a $5$-recollement
of $\mathcal{D}(Mod)$-level, then, there is a recollement
of $\underline{\Gproj }$-level (ref. Proposition~\ref{proposition-4-recoll}).
However, we get a contradiction by analyzing the
Gorenstein projective modules, as \cite[Example 3.16]{Lu17}.}

\end{example}

\begin{example}
{\rm Let $k$ be a field and $Q$ be the following quiver
$$\xymatrix{4 \ar@<-0.5ex>[r]_{\beta _1} & 5
\ar@<-0.5ex>[l]_{\alpha _1} \ar@<-0.5ex>[r]_{\beta _2}
& 6 \ar[r]_{\beta _3} \ar@<-0.5ex>[l]_{\alpha _2}
& 7  \\  1 \ar@<-0.5ex>[u]_{\gamma _3} \ar@<+0.5ex>[u]^{\gamma _1} \ar@<-0.5ex>[r]_{\delta _1} & 2
\ar@<-0.5ex>[l]_{\theta _1} \ar[u]_{\gamma _2} & 3 \ar[l]_{\theta _2}}.$$
Let $A=kQ/I$ where $I$ is the ideal of $kQ$ generated by
$\alpha _1 \beta _1$, $\beta _1
\alpha _1$, $\alpha _2 \beta _2$, $\beta _2
\alpha _2$, $\beta _2 \beta _3$,
$\theta _1 \delta _1$, $\delta _1
\theta _1$, $\theta _2 \theta _1$, $\theta _2 \gamma _2$,
$\gamma _1 \beta _1- \delta _1 \gamma _2 $ and $\theta _1 \gamma _1
-\gamma _2\alpha _1$. Then $A = \left[\begin{array}{cc} B & 0
\\ M & C \end{array}\right] $,
where $B=(1-e)A(1-e)$, $C=eAe$, $e=e_1+e_2+e_3$.
Obviously, $\pd  _CM=1$ and $M_B$ is projective.
By \cite[Example 3.4]{AKLY17},
$\mathcal{D}A$ admits a $4$-recollement relative
to $\mathcal{D}C$ and $\mathcal{D}B$.

Clearly, both $B$ and $C$ are representation-finite, and
thus they satisfy the Auslander-Reiten
conjecture. On the other hand, it is easy to
see that $\underline{^\bot C}=0$.
By Theorem~\ref{theorem-4}, the Auslander-Reiten
conjecture is true for $A$, and so does the Gorenstein projective
conjecture.

We mention that the algebra $A$ doesn't belong to
one of the classes of algebras where the Auslander-Reiten
conjecture is known to
hold, e.g. algebras of finite representation type
and symmetric algebras with radical cube zero.}

\end{example}
\bigskip

\noindent {\footnotesize {\bf ACKNOWLEDGMENT.} This work is supported by
the National Natural Science Foundation of China (11701321, 11601098) and Yunnan Applied Basic Research
Project 2016FD077.}

\end{document}